\documentclass[12pt,a4paper,reqno]{article}
\usepackage[T1]{fontenc}
\usepackage{lmodern}
\usepackage[a4paper,left=25mm,top=28.4mm,textwidth=160mm,textheight=225mm,
headheight=0pt,headsep=0pt,footskip=12mm]{geometry}
\usepackage{amsmath,amssymb,amsthm,mathtools}
\usepackage{booktabs,array,graphicx,microtype,needspace}
\usepackage{indentfirst,xcolor,capt-of}
\usepackage[colorlinks=true,linkcolor=red,anchorcolor=green,
citecolor=blue,urlcolor=blue,bookmarksopen=true,
bookmarksnumbered=true,bookmarksdepth=2]{hyperref}
\hypersetup{pdftitle={New result on Lane-Emden conjecture with exponents in a disk tangent to the critical hyperbola},
pdfauthor={Haoyang Lu; Zhitao Zhang},
pdfsubject={Liouville theorem with exact rational positivity certificates}}
\numberwithin{equation}{section}
\theoremstyle{plain}
\newtheorem{theorem}{Theorem}[section]
\newtheorem{lemma}[theorem]{Lemma}
\newtheorem{proposition}[theorem]{Proposition}
\newtheorem*{laneemdenconjecture}{Conjecture (Lane-Emden)}
\newtheorem{remark}[theorem]{Remark}

\newcommand{\R}{\mathbb R}
\newcommand{\dd}{\,\mathrm dx}
\newcommand{\Div}{\operatorname{div}}
\newcommand{\tr}{\operatorname{tr}}
\newcommand{\norm}[1]{\lVert#1\rVert}
\newcommand{\proofstep}[1]{%
  \par\medskip\Needspace{3\baselineskip}%
  \noindent\textbf{#1}\enspace\ignorespaces}
\allowdisplaybreaks[2]
\title{\large\textbf{New result on Lane-Emden conjecture
with exponents in a disk}\thanks{
Z. Zhang is supported by the National Key R\&D Program of China
(2022YFA1005601) and the National Natural Science Foundation of China
(12671146,12031015).}}
\author{Haoyang Lu$^{a,b}$~~Zhitao Zhang$^{a,b}$\\
{\small $^{a}$ State Key Laboratory of Mathematical Sciences, Academy of Mathematics and Systems}\\
{\small Science, the Chinese Academy of Sciences, Beijing 100190, P. R. China}\\
{\small $^{b}$ School of Mathematical Sciences, University of Chinese Academy of Sciences,}\\
{\small Beijing 100049, P. R. China}\\
{\small E-mails: luhaoyang25@mails.ucas.ac.cn (H. Lu), zzt@math.ac.cn (Z. Zhang)}}
\date{}

\begin{document}
\maketitle
\begin{abstract}
We obtain new results on Lane-Emden conjecture to prove a Liouville theorem for the Lane-Emden system
$$\begin{cases}
-\Delta u=v^p, & x\in\R^N,\\
-\Delta v=u^q, & x\in\R^N,
\end{cases}
$$ with $N\ge3$,
$p,q>0$, $pq>1$, and
\[
\left(p-\frac2{N-2}\right)^2+
\left(q-\frac2{N-2}\right)^2\le\frac{2N^2}{(N-2)^2},
\]
the system has no positive classical solution as $(p,q)$ in the disk unless $p=q=(N+2)/(N-2)$, no growth or decay assumptions are imposed, and the disk is tangent to the critical hyperbola. Especially, for $N\ge5$,
the disk contains a nonempty open set  not covered by the criteria of
Busca-Man\'asevich~\cite{BM} and Souplet~\cite{Souplet}, or by the
explicit conditions of Li-Li-Wei~\cite{LLW}.
After reducing the exponent range by known Liouville criteria, we first construct vector fields to
derive a divergence inequality, then apply Sobolev embedding theorem,
Gagliardo-Nirenberg interpolation inequality, and Young's inequality to give
a uniform local integral bound
$\int_{B_1}(u^rv^{2p}+u^{2q}v^s)\dd\le C(N,p,q)$
for suitable weights  for any positive solution $(u,v)$, finally the nonexistence of positive classical solutions follows by scaling.
\\[0.3em]
\textbf{Keywords:} Lane-Emden system, Liouville theorem,
Picone identity, divergence identity, computer-assisted proof.\\[0.3em]
\textbf{AMS Subject Classification (2020):} 35J47, 35B53, 35B09.
\end{abstract}

\section{Introduction}\label{sec:introduction}

We study the nonexistence of positive entire solutions of the
Lane-Emden system
\begin{equation}\label{eq:system}
\begin{cases}
-\Delta u=v^p, & x\in\R^N,\\
-\Delta v=u^q, & x\in\R^N,
\end{cases}
\end{equation}
where $N\ge3$ and $p,q>0$. By a positive entire solution we mean
a pair $u,v\in C^2(\R^N)$ satisfying \eqref{eq:system}, with
$u,v>0$.

Liouville theorems for the system  \eqref{eq:system} are used to obtain a priori
bounds for elliptic boundary value problems; see
Cl\'ement, de Figueiredo, and Mitidieri~\cite{CFM}.
The expected critical criterion for existence or not is the critical Sobolev hyperbola
\begin{equation}\label{eq:critical}
\frac1{p+1}+\frac1{q+1}=\frac{N-2}{N}.
\end{equation}

A pair $(p,q)$ is called subcritical if it lies below this hyperbola,
which means the left-hand side of \eqref{eq:critical} is strictly larger
than the right-hand side. Figure~\ref{fig:conjecture} shows the
critical hyperbola together with the curve $pq=1$.
We find that the equation~\eqref{eq:critical} can also be written as
\begin{equation}\label{L-Z-9-21}
\left(p-\frac2{N-2}\right)\left(q-\frac2{N-2}\right)
=\frac{N^2}{(N-2)^2}.
\end{equation}
\begin{laneemdenconjecture}
Let $N\ge3$ and $p,q>0$. If
\begin{equation}\label{eq:subcritical}
\frac1{p+1}+\frac1{q+1}>\frac{N-2}{N},
\end{equation}
then system \eqref{eq:system} has no positive classical solution.
\end{laneemdenconjecture}

\par\smallskip\noindent
\begin{minipage}{\textwidth}
\centering
\includegraphics[width=0.66\textwidth,trim=0pt 14pt 0pt 3pt,clip]{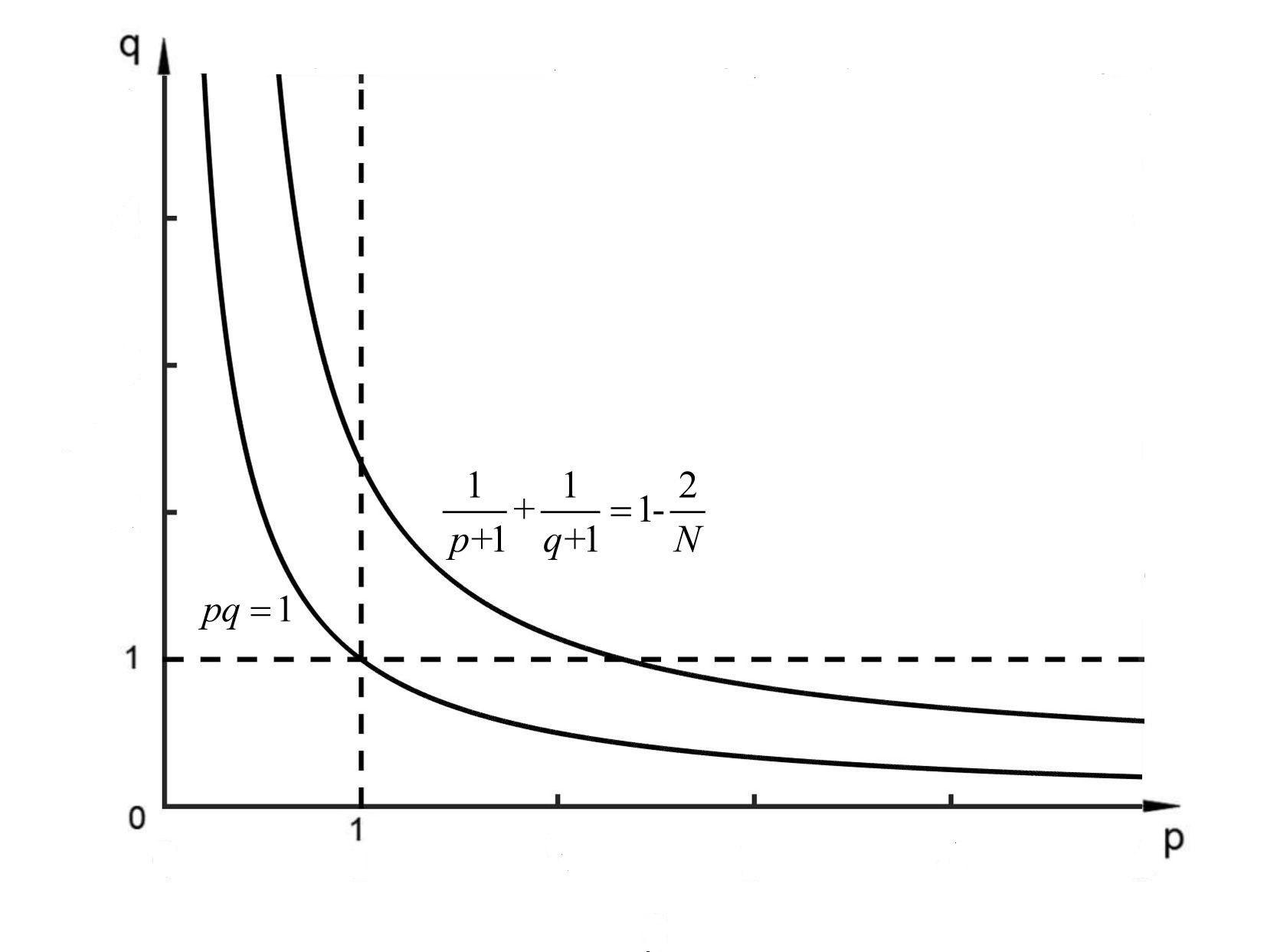}\par
\begingroup
\let\normalsize\footnotesize
\setlength{\abovecaptionskip}{4pt}
\captionof{figure}{The critical Sobolev hyperbola and $pq=1$. %($N\ge5$).
Dashed lines mark $p=1$ and $q=1$.}
\label{fig:conjecture}
\endgroup
\end{minipage}
\par\medskip

In this paper we write
\[
P=\frac{N+2}{N-2}.
\]
When $pq\neq 1$, the scaling exponents are
\begin{equation}\label{eq:scaling}
\alpha=\frac{2(p+1)}{pq-1},\qquad
\beta=\frac{2(q+1)}{pq-1}.
\end{equation}
If $pq\neq 1$ and $(u,v)$ is a positive solution of \eqref{eq:system},
then $(R^\alpha u(Rx),R^\beta v(Rx))$ solves the same system for
every $R>0$. When $pq>1$, condition \eqref{eq:subcritical} is
equivalent to $\alpha+\beta>N-2$.

For the scalar equation $-\Delta w=w^\ell, x\in\R^N$, with $1<\ell<\frac{N+2}{N-2}$,
Gidas and Spruck~\cite{GS} proved nonexistence of positive entire
solutions by differential identities and weighted integral
estimates. Chen and Li~\cite{ChenLi91}
gave another proof using the Kelvin transform and moving planes.

For the system \eqref{eq:system}, Mitidieri~\cite{MitidieriRellich}
introduced a Rellich-type identity. De Figueiredo and Felmer~\cite{FF}
used the Kelvin transform and moving planes to prove nonexistence
for $0<p,q\le P$, $(p,q)\ne(P,P)$.
Mitidieri~\cite{Mitidieri96} proved nonexistence of positive radial
solutions as $(p,q)$ is below the critical hyperbola using integral identities.
Serrin and Zou~\cite{SerrinZou96} established
nonexistence of positive solutions for $pq\le1$, or $pq>1$ with
$\max\{\alpha,\beta\}\ge N-2$. In dimension three, they also proved the conjecture for $(p,q)$
in the full subcritical range with the added restriction that solutions have polynomial growth,
using integral estimates and a Rellich-Pohozaev identity. Their radial
constructions in 1998~\cite{SerrinZou98} show that positive entire
solutions exist for $(p,q)$ on and above the critical hyperbola.

Reichel and Zou~\cite{ReichelZou} recovered nonexistence in the range
$0<p,q\le P$, $(p,q)\ne(P,P)$ by moving spheres.
Busca and Man\'asevich (2002)~\cite{BM} proved nonexistence of positive solutions when
\[
\alpha,\beta\in\left[\frac{N-2}{2},N-2\right),\qquad
(\alpha,\beta)\ne\left(\frac{N-2}{2},\frac{N-2}{2}\right).
\]
Their proof uses moving planes and maximum principles after
transforming the system to the cylinder $\R\times S^{N-1}$.
Together with the preceding classical criterion, this covers
$\min\{\alpha,\beta\}\ge(N-2)/2$, except for the critical
diagonal pair.

In low dimensions, Pol\'a\v cik, Quittner, and Souplet~\cite{PQS}
used the doubling lemma and rescaling to reduce the problem for
$p,q>1$ to bounded entire solutions. This removed the growth
assumption from the three-dimensional argument and proved the
conjecture in $\mathbb{R}^3$. Souplet~\cite{Souplet} established
nonexistence under \eqref{eq:subcritical} and
$\max\{\alpha,\beta\}>N-3$, hence he proved the conjecture for $N=4$. His proof combines a component
comparison, a Rellich-Pohozaev identity, Sobolev interpolation
on spheres, and a feedback argument.
Chen and Li (2009)~\cite{ChenLi} used moving planes for the integral
formulation to obtain radial symmetry and subcritical nonexistence of positive solutions
under additional integrability assumptions. Cheng and Huang
(2019)~\cite{ChengHuang} used integral estimates to characterize
the conjecture by an energy-type a priori bound.

More recently, K. Li, M. Li, and J. Wei~\cite{LLW} combined an Obata-type integral
inequality, Picone's identity and scaling to prove nonexistence
for $N\ge5$, $p,q\ge1$, and
\begin{equation}\label{eq:LLW}
\frac1{p+1}+\frac1{q+1}-\frac{N-2}{N}\ge\frac4{N^2}.
\end{equation}
Their Theorem~3.2 replaces the number $4$ in \eqref{eq:LLW} by any fixed $c_0>2$, when
$N$ is sufficiently large depending on $c_0$.
% BEGIN HLE HISTORY -- based on Shantou-2026-1-19, Part (A).

Related work concerns the H\'enon-Lane-Emden system, with weights
$|x|^a$ and $|x|^b$ in the two equations of \eqref{eq:system},
considered away from the origin,
$$\begin{cases}
-\Delta u=|x|^a v^p, & x\in\R^N,\\
-\Delta v=|x|^b u^q, & x\in\R^N.
\end{cases}
$$
Its scalar form goes back to H\'enon's 1973 stellar
model~\cite{Henon73}. For $a,b>-2$, the weighted subcritical
condition is $(N+a)/(p+1)+(N+b)/(q+1)>N-2$, and positive solutions are
taken in \mbox{$C^2(\R^N\setminus\{0\})\cap C(\R^N)$}.
Bidaut-V\'eron and Giacomini (2010)~\cite[Theorem 1.4]{BVG10}
proved that, for $pq>1$, positive radial entire solutions exist
precisely on or above the weighted critical hyperbola.
Phan and Souplet (2012)~\cite{PhanSouplet12} proved
the three-dimensional scalar conjecture for bounded positive
solutions. Fazly and Ghoussoub (2014)~\cite{FazlyGhoussoub14}
extended this result to bounded solutions of the system with
nonnegative weights and also treated stable solutions in higher
dimensions. Kui Li and Zhitao Zhang (2019)~\cite{LiZhang19} proved
the full weighted conjecture in $\R^3$ by weighted energy
estimates, Sobolev inequalities on $S^2$, and scaling, without
boundedness assumptions or decay assumptions at infinity on positive solutions $(u,v)$.

Li and Zhang (2021)~\cite{LiZhang21} proved Liouville theorems for
higher-order Hardy-H\'enon systems and their integral formulations.
In 2022, Cheng, Li, and
Zhang~\cite{ChengLiZhang22} obtained Liouville theorems for
generalized weighted divergence-form systems in $\R^2$ and for
bounded positive solutions in $\R^3$, under the corresponding parameter
conditions. H. Li (2022)~\cite{HangLi22} proved nonexistence of positive solutions
for the weighted system when $a,b>-2$, $pq>1$, the weighted
subcritical condition holds, and $p,q<4/3$ for $N=4$ or
$p,q<10/9$ for $N=5$. Li, Yu, and Zhang (2024)~\cite{LiYuZhang24}
treated fractional Hardy-H\'enon systems by integral estimates
and scaling spheres to obtain Liouville theorems.
% END HLE HISTORY
\par
Our main result concerns the disk centered at
$(2/(N-2),2/(N-2))$ with radius \mbox{$\sqrt2\,N/(N-2)$}, whose boundary
is tangent to the critical hyperbola \eqref{eq:critical} at $(P,P)$.
% Keep Theorem 1.1 and its schematic together, with no intervening text.
\par\noindent
\begin{minipage}{\textwidth}
\begin{theorem}\label{thm:main}
Let $N\ge3$ be an integer, and let $p,q>0$ satisfy $pq>1$. If
\begin{equation}\label{eq:disk}
\left(p-\frac2{N-2}\right)^2+
\left(q-\frac2{N-2}\right)^2
\le\frac{2N^2}{(N-2)^2},
\qquad (p,q)\ne(P,P),
\end{equation}
then \eqref{eq:system} has no positive classical solution.
\end{theorem}
\par\smallskip
\begingroup
\centering
\includegraphics[width=0.86\textwidth]{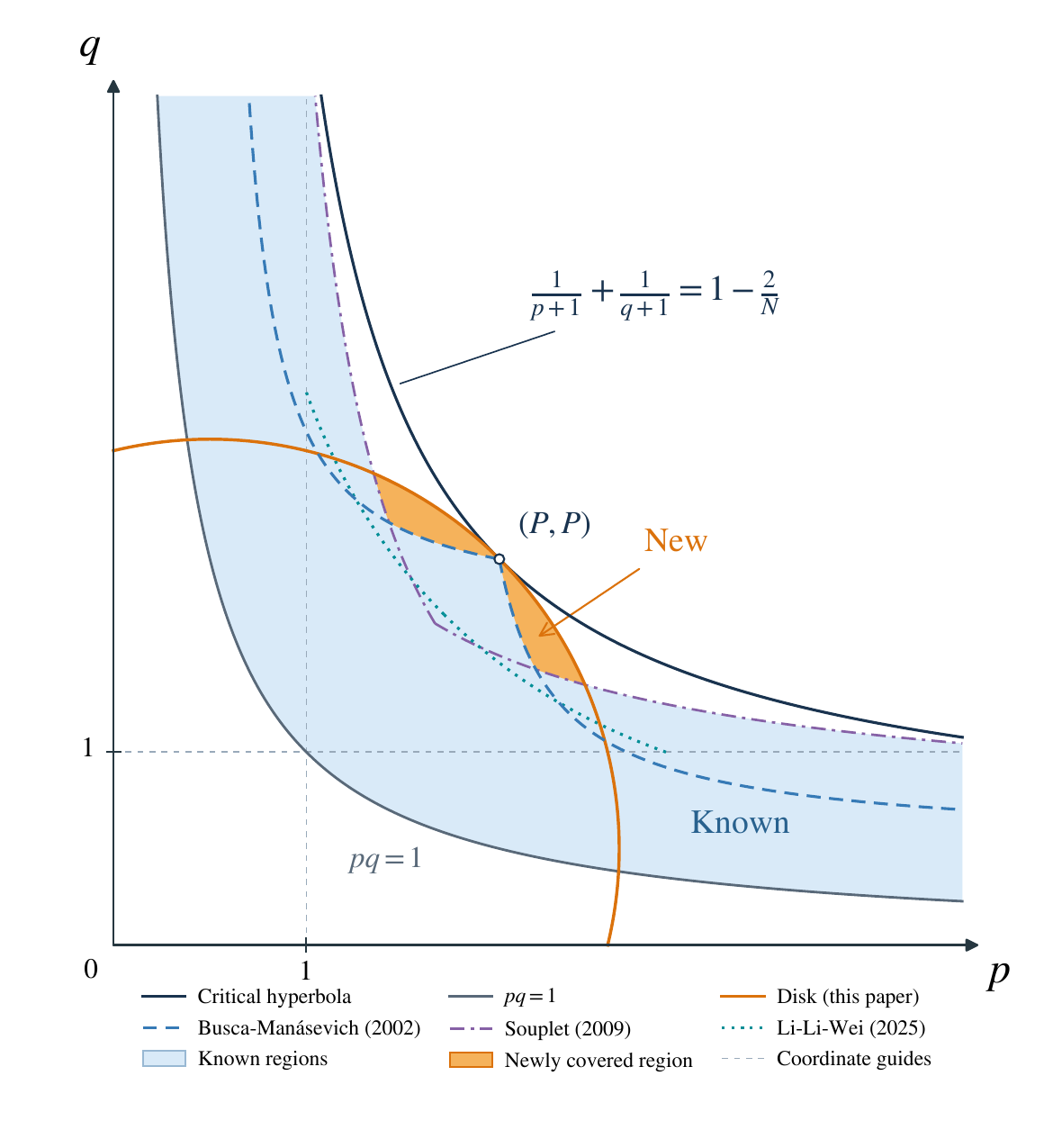}\par
% Local caption size: 10pt (article class at 12pt).
\let\normalsize\footnotesize
\captionof{figure}{The dark solid curve is the critical hyperbola,
the grey solid curve is $pq=1$,
and the orange solid curve is the disk boundary. Blue dashed, purple
dash-dotted and teal dotted curves represent
Busca-Man\'asevich~\cite{BM}, Souplet~\cite{Souplet} and
Li-Li-Wei~\cite{LLW}, respectively. Light blue indicates known
regions and orange the new region.}
\label{fig:regions}
\endgroup
\end{minipage}
\par\medskip

\Needspace{8\baselineskip}

\proofstep{Outline of the proof of Theorem \ref{thm:main}.}
\par
Set
\begin{equation}\label{eq:disk-margin}
G_N=\frac{2N^2}{(N-2)^2}
-\left(p-\frac2{N-2}\right)^2
-\left(q-\frac2{N-2}\right)^2,
\end{equation}
and, for $N\ge5$, define
\begin{equation}\label{eq:remaining}
\mathcal R_N=
\left\{(p,q):
\sqrt P\le q<P<p,\quad
q(p-P+1)\ge P,\quad G_N\ge0
\right\},
\end{equation}
where every pair in $\mathcal R_N$ satisfies $p>q>1$.
We will show symmetry and known Liouville criteria reduce the proof to only for
$N\ge5$ and $(p,q)\in\mathcal R_N$.

The main step is to construct vector fields $J_u,J_v$ (here $(u,v)$ is the solution of  \eqref{eq:system}) such that the divergence inequality is satisfied
\[
\operatorname{div}(J_u+\omega J_v)
\ge\varepsilon_0
\left(
u^{r-2}|\nabla u|^4
+v^{s-2}|\nabla v|^4
+u^rv^{2p}+u^{2q}v^s
\right),
\]
here $-2<r,s<0$ and $\omega,\varepsilon_0>0$, and all parameters
depend only on $N,p,q$. We choose the weights in order to give  the two
zeroth-order terms the same scaling degree. The coefficient
inequalities are needed to combine the two fields which will be verified by
elementary estimates and exact Bernstein coefficient calculations.

Testing this inequality against a cutoff function, we use the positive
quartic terms to absorb the cutoff errors. The local $L^1$ bound,
Sobolev embedding theorem, Gagliardo-Nirenberg interpolation inequality, 
and Young's inequality give an estimate
that can be iterated over concentric balls with increasing radii. This yields the uniform bound for any positive solutions of \eqref{eq:system}
\[
\int_{B_1}(u^rv^{2p}+u^{2q}v^s)\dd\le C(N,p,q).
\]
Both integrands have the same scaling degree $\nu>N$. Applying
the bound to the scaling solution $(R^\alpha u(Rx),R^\beta v(Rx))$ gives
\[
\int_{B_R}\bigl(u^rv^{2p}+u^{2q}v^s\bigr)\,dx
\le C(N,p,q)R^{N-\nu}\longrightarrow0
\qquad(R\to\infty).
\]
This contradicts the fact that the solution over $\mathbb{R}^N$ is positive..

\proofstep{Organization of the paper.}

Section~\ref{sec:reduction} reduces the exponent range by known
Liouville theorems. Section~\ref{sec:currents} proves the divergence
estimate: Subsection~\ref{sec:identity-combination} derives and combines
the component identities, and Subsection~\ref{sec:parameter-positivity}
chooses the parameters and verifies the coefficient inequalities.
Section~\ref{sec:local} obtains the local integral bound and
completes the proof of Theorem~\ref{thm:main} by scaling.
 Appendix~\ref{app:data} records the computed Bernstein
coefficient data.

\section{Reduction of the exponent range}\label{sec:reduction}

By symmetry and known Liouville theorems, it suffices to prove
Theorem~\ref{thm:main} for $N\ge5$ and $(p,q)\in\mathcal R_N$.
We also show that $\mathcal R_N$ has nonempty interior.

\subsection{Known Liouville criteria}

 We recall two Liouville theorems used in the reduction.

\begin{proposition}
\label{prop:Souplet}
Let $N\ge3$, $p,q>0$, and $pq>1$. Assume
\[
\frac1{p+1}+\frac1{q+1}>\frac{N-2}{N}.
\]
Then \eqref{eq:system} has no positive entire solution if either
$N=3,4$, or
\[
N\ge5,\qquad \max\{\alpha,\beta\}>N-3.
\]
\end{proposition}

The case $N=3$ follows from Serrin-Zou~\cite{SerrinZou96} and
Pol\'a\v cik-Quittner-Souplet~\cite{PQS}; the cases $N=4$ and
$N\ge5$ are due to Souplet~\cite[Theorems 1 and 2]{Souplet}.

\begin{proposition}[Busca-Man\'asevich {\cite[Theorem 2.1]{BM}}]
\label{prop:BM}
Suppose $N\ge3$, $p,q>0$, $pq>1$, and
\[
\alpha,\beta\in[(N-2)/2,N-2),\qquad
(\alpha,\beta)\ne((N-2)/2,(N-2)/2).
\]
Then \eqref{eq:system} has no positive entire solution.
\end{proposition}

\begin{remark}
Strict subcriticality is also required when applying Souplet's theorem.
\end{remark}

\subsection{Reduction to the remaining region}

We use symmetry and the classical criteria to 
%reduce Theorem~\ref{thm:main} to $\mathcal R_N$  
show that we only need to give the proof
of Theorem~\ref{thm:main}  for $N\ge5$ and
$(p,q)\in\mathcal R_N$. We first verify strict
subcriticality throughout the disk, except at the excluded pair $(P,P)$.

\begin{proposition}[Reduction to $\mathcal R_N$]\label{prop:reduction}
To prove Theorem~\ref{thm:main}, it suffices to prove nonexistence
for $N\ge5$ and $(p,q)\in\mathcal R_N$.
\end{proposition}

\begin{proof}
\proofstep{Step 1: Strict subcriticality and classical criteria.}

Proposition~\ref{prop:Souplet} requires strict subcriticality.
We verify this condition for the exponent pairs in
Theorem~\ref{thm:main}.
Since $2/(N-2)=(P-1)/2$ and $N/(N-2)=(P+1)/2$,
expanding the squares in \eqref{eq:disk-margin} gives
\[
\begin{aligned}
G_N&=\frac{(P+1)^2}{2}
-\left(p-\frac{P-1}{2}\right)^2
-\left(q-\frac{P-1}{2}\right)^2\\
&=\frac{(P+1)^2-(P-1)^2}{2}
+(P-1)(p+q)-p^2-q^2\\
&=2P+(P-1)(p+q)-p^2-q^2.
\end{aligned}
\]
Adding $(p-q)^2=p^2-2pq+q^2$ to both sides gives
\[
G_N+(p-q)^2=2P+(P-1)(p+q)-2pq.
\]
Putting the subcriticality gap over a common denominator gives
\begin{equation}\label{eq:gap}
\begin{aligned}
\frac1{p+1}+\frac1{q+1}-\frac{N-2}{N}
&=\frac{N(p+q+2)-(N-2)(pq+p+q+1)}
        {N(p+1)(q+1)}\\
&=\frac{2(p+q)+N+2-(N-2)pq}{N(p+1)(q+1)}\\
&=\frac{(N-2)[G_N+(p-q)^2]}{2N(p+1)(q+1)}.
\end{aligned}
\end{equation}
The last equality uses $(N-2)P=N+2$ and $(N-2)(P-1)=4$.
For $N\ge3$ and $p,q>0$, the denominator and the factor $N-2$
in \eqref{eq:gap} are positive. Since $G_N\ge0$ in the disk,
the gap is nonnegative. Equality would require $p=q$ and
$G_N=0$. But
\[
G_N(p,p)=2(P-p)(p+1),
\]
so $p>0$ forces $p=P$. The critical pair $(P,P)$ is explicitly
excluded by the condition $(p,q)\ne(P,P)$ in
Theorem~\ref{thm:main}. Hence every exponent pair allowed by
the theorem is strictly subcritical.

Since strict subcriticality has now been established,
Proposition~\ref{prop:Souplet} gives nonexistence for $N=3,4$
and, for $N\ge5$, whenever $\max\{\alpha,\beta\}>N-3$.
We may therefore restrict the rest of the proof to
\begin{equation}\label{eq:after-Souplet}
N\ge5,\qquad \max\{\alpha,\beta\}\le N-3<N-2.
\end{equation}
In particular, $\alpha,\beta<N-2$, which is the upper bound
required in Proposition~\ref{prop:BM}.

\proofstep{Step 2: Reduction to $p>q$.}

If $p=q$, then $pq>1$ gives $p>1$, while
$G_N(p,p)=2(P-p)(p+1)\ge0$ and the excluded critical pair give
$p<P$. Hence
\[
\alpha=\beta=\frac2{p-1}>\frac2{P-1}=\frac{N-2}{2}.
\]
Since \eqref{eq:after-Souplet} also gives $\alpha,\beta<N-2$,
Proposition~\ref{prop:BM} applies.

We may therefore assume $p\ne q$. Interchanging
$(u,v,p,q)$ with $(v,u,q,p)$ if necessary, we take $p>q$.
From \eqref{eq:scaling} and $pq>1$,
\[
\alpha-\beta=\frac{2(p-q)}{pq-1}>0.
\]
We next show that $q<P$. From the expression for $G_N$ in
Step~1,
\[
\begin{aligned}
2G_N+(p-q)^2
&=4P+2(P-1)(p+q)-2(p^2+q^2)+(p-q)^2\\
&=4P+2(P-1)(p+q)-p^2-2pq-q^2\\
&=4P+2(P-1)(p+q)-(p+q)^2\\
&=(2P-p-q)(p+q+2).
\end{aligned}
\]
Here $G_N\ge0$ by the disk condition and $(p-q)^2>0$ because
$p>q$. Thus the left-hand side is strictly positive.
Since $p+q+2>0$, we obtain $p+q<2P$, and hence
$q<(p+q)/2<P$.

\proofstep{Step 3: The Busca-Man\'asevich region and the case $q\le1$.}

Set $p_B(q)=P-1+P/q$. To express the threshold
$\beta\ge(N-2)/2$ in terms of $p,q$, compute
\begin{equation}\label{eq:BM-gap}
\begin{aligned}
\beta-\frac{N-2}{2}
&=\frac{4(q+1)-(N-2)(pq-1)}{2(pq-1)}\\
&=\frac{(N-2)[(P-1)q+P-pq]}{2(pq-1)}\\
&=\frac{(N-2)q[p_B(q)-p]}{2(pq-1)}.
\end{aligned}
\end{equation}
The factors outside $p_B(q)-p$ are positive. Thus
$p\le p_B(q)$ is equivalent to $\beta\ge(N-2)/2$.
In that case,
\[
\frac{N-2}{2}\le\beta<\alpha\le N-3<N-2,
\]
and Proposition~\ref{prop:BM} applies.

This includes $0<q\le1$. Indeed,
\begin{equation}\label{eq:small-exponent-coverage}
\begin{aligned}
p_B(q)-(2P-q)
&=P-1+\frac Pq-2P+q\\
&=\frac{q^2-(P+1)q+P}{q}
=\frac{(q-1)(q-P)}q\ge0
\qquad(0<q\le1).
\end{aligned}
\end{equation}
Here $q-1\le0$ and $q-P<0$. Since $p+q<2P$, we obtain
$p<2P-q\le p_B(q)$, as required.

\proofstep{Step 4: Reduction to the coefficient region $\mathcal R_N$.}
The remaining case has $p>p_B(q)$ and $q>1$. Since $q<P$,
\[
p_B(q)-P=\frac{P-q}{q}>0,
\]
so $q<P<p$. The disk inequality is
\[
-G_N(p,q)=p^2+q^2-(P-1)(p+q)-2P\le0.
\]
Its value at $p=p_B(q)$ is smaller, because
\[
-G_N(p,q)+G_N(p_B(q),q)
=(p-p_B(q))[p+p_B(q)-(P-1)]>0.
\]
Hence $-G_N(p_B(q),q)<0$.

To calculate this boundary value, first note that
\[
p_B(q)^2-(P-1)p_B(q)
=p_B(q)\frac Pq
=\frac{P(P-1)}q+\frac{P^2}{q^2}.
\]
Substitution and multiplication by $q^2$ give
\begin{equation}\label{eq:boundary-factor}
\begin{aligned}
q^2[-G_N(p_B(q),q)]
&=q^4-(P-1)q^3-2Pq^2+P(P-1)q+P^2\\
&=(q^2-P)[q^2-(P-1)q-P]\\
&=(q-P)(q^2-P)(q+1)<0.
\end{aligned}
\end{equation}
Because $q-P<0$ and $q+1>0$, this implies $q^2>P$, so
$q>\sqrt P$. Finally, $p>p_B(q)$ yields
\[
q(p-P+1)>q[p_B(q)-P+1]=P.
\]
Thus every remaining pair satisfies
\[
\sqrt P<q<P<p,\qquad q(p-P+1)>P,\qquad G_N\ge0,
\]
and hence belongs to $\mathcal R_N$ as defined in
\eqref{eq:remaining}. This proves the reduction.
\end{proof}

\subsection{Nonemptiness of the remaining region}

\begin{lemma}\label{lem:region-nonempty}
For every \(N\ge5\), the set \(\mathcal R_N\) has nonempty interior.
\end{lemma}

\begin{proof}
Let
\[
p_B(q)=P-1+\frac Pq.
\]
For every \(q\in(\sqrt P,P)\),
\[
p_B(q)-P=\frac{P-q}{q}>0,\qquad
p_B(q)-q=\frac{(P-q)(q+1)}q>0
\]
and
\[
p_B(q)^2+q^2-(P-1)[p_B(q)+q]-2P
=\frac{(q-P)(q^2-P)(q+1)}{q^2}<0.
\]
Thus \((p_B(q),q)\) lies strictly inside the disk, with
\(q<p_B(q)\), and it satisfies
\(q[p_B(q)-P+1]=P\). By continuity, all \(p>p_B(q)\)
sufficiently close to \(p_B(q)\) satisfy
\[
q<p,\qquad q(p-P+1)>P,\qquad G_N>0.
\]
These points form a nonempty open subset of \(\mathcal R_N\).
\end{proof}

\begin{remark}
For every $N\ge5$, the disk contains a nonempty open set of
exponents $p>q>1$, which is not covered by the Busca-Man\'asevich,
Souplet, or explicit Li-Li-Wei conditions recalled above.
To see this, take $t>0$ sufficiently small and set
\[
p_t=\frac{2}{N-2}+\frac{N}{N-2}\frac{1+2t-t^2}{1+t^2},
\qquad
q_t=\frac{2}{N-2}+\frac{N}{N-2}\frac{1-2t-t^2}{1+t^2}.
\]
This pair lies on the circle and satisfies $p_t>P>q_t>1$.
As $t\downarrow0$,
\[
\begin{gathered}
\beta_t-\frac{N-2}{2}=-\frac{N-2}{2}t+O(t^2),
\qquad \alpha_t\longrightarrow\frac{N-2}{2}<N-3,\\
\frac1{p_t+1}+\frac1{q_t+1}-\frac{N-2}{N}
=\frac{2(N-2)t^2}{N(1-t^2)}.
\end{gathered}
\]
Thus, for $t>0$ sufficiently small, we have
$\beta_t<(N-2)/2$, $\max\{\alpha_t,\beta_t\}<N-3$, and the positive
subcriticality gap is less than $2/N^2$. These strict
inequalities exclude all three criteria and persist in a
neighborhood of $(p_t,q_t)$. Intersecting this neighborhood with
the interior of the disk gives the required open set.
\end{remark}

\Needspace{6\baselineskip}
\section{A divergence inequality}\label{sec:currents}

Throughout this section, let $N\ge5$ and $p,q>1$.

In Subsection~\ref{sec:identity-combination}, we derive the component
identities and a criterion for combining them. In
Subsection~\ref{sec:parameter-positivity}, we choose the parameters
and verify the coefficient inequalities.

\begin{proposition}[Divergence inequality]
\label{prop:circle-coefficients}
Let $(p,q)\in\mathcal R_N$. Set
\begin{equation}\label{eq:chosen-weights}
\tilde{\kappa}=\frac{N-3}{N},\qquad
r=\tilde{\kappa}(q+1)-2,\qquad s=\tilde{\kappa}(p+1)-2.
\end{equation}

Then $-2<r,s<0$; there are vector fields $J_u,J_v$ and constants $\omega,\varepsilon_0>0$
such that every positive classical solution  $(u,v)$ of \eqref{eq:system} satisfies
\begin{equation}\label{eq:constructed-coercivity}
\Div(J_u+\omega J_v)\ge\varepsilon_0\bigl(
u^{r-2}|\nabla u|^4+v^{s-2}|\nabla v|^4
+u^rv^{2p}+u^{2q}v^s\bigr).
\end{equation}
All coefficients (to be defined in proposition \ref{prop:criterion}) and constants $\omega, \varepsilon_0$  depend only on $N,p,q$.
\end{proposition}
Here we  introduce the vector fields of the
form \eqref{eq:current-general}
\begin{equation}\label{10-501}
%\begin{aligned}
J_u=u^rD^2u\,\nabla u
+\left(k_u-\frac r2\right)u^{r-1}|\nabla u|^2\nabla u
+\left(\frac1N-b_u\right)u^rv^p\nabla u
+c_u u^{r+1}v^{p-1}\nabla v,
\end{equation}
\begin{equation}\label{10-502}
J_v=v^sD^2v\,\nabla v
+\left(k_v-\frac s2\right)v^{s-1}|\nabla v|^2\nabla v
+\left(\frac1N-b_v\right)v^su^q\nabla v
+c_v v^{s+1}u^{q-1}\nabla u.
%\end{aligned}
\end{equation}

The proof occupies this section. We first derive the component
identity and a criterion for combining the two fields, then
choose the parameters and verify each of the required signs.

\subsection{Divergence identities and their combination}\label{sec:identity-combination}

\subsubsection{The Obata-Picone divergence identity}

We first express the divergence associated with one component as
a sum of square terms, quartic gradient terms and zeroth-order
coupling terms. Interchanging the components gives the second
identity needed below.

For a $N \times N$  symmetric matrix $H$, write
\[
H^\circ=H-\frac{\tr H}{N}I,\qquad
|H^\circ|^2=|H|^2-\frac{(\tr H)^2}{N}.
\]
The elementary Picone identity
\[
|\nabla\phi|^2-
\nabla\left(\frac{\phi^2}{w}\right)\cdot\nabla w
=\left|\nabla\phi-\frac{\phi}{w}\nabla w\right|^2,
\qquad w>0,
\]
with $w=u^\theta$ ($\theta\in\R$) and $\phi=v$, will give the gradient square
in \eqref{eq:identity}.

Fix $r,k,\theta\in\R$. Put $X=\nabla u$, $Y=\nabla v$, and define
the vector field
\begin{equation}\label{eq:current-general}
\begin{aligned}
J={}&u^rD^2u\,X+
\left(k-\frac r2\right)u^{r-1}|X|^2X\\
&+\left(\frac1N-b\right)u^rv^pX
+c\,u^{r+1}v^{p-1}Y.
\end{aligned}
\end{equation}

\begin{remark}The coefficient $k-r/2$ makes the Hessian cross term in $\Div J$
equal to $2ku^{r-1}D^2u[X,X]$.
\end{remark}

\Needspace{6\baselineskip}
\begin{lemma}[Obata-Picone divergence identity]\label{lem:identity}
Set $c_N=(N-1)/N$ and
\[
\lambda=\frac32r-\frac{N+2}{N}k,\qquad
\tau=r+1+2(p-1)\theta,\qquad
\delta=r\tau+p(p-1)\theta^2.
\]
If $\delta\ne0$ and
\begin{equation}\label{eq:bc}
b=\frac{\lambda\tau}{\delta}-c_N,\qquad
c=\frac{\lambda p}{\delta},
\end{equation}
then the vector field \eqref{eq:current-general} satisfies
\begin{equation}\label{eq:identity}
\begin{aligned}
\Div J={}&u^r\left|\left(D^2u+\frac kuX\otimes X\right)^\circ\right|^2
+A(r,k)u^{r-2}|X|^4\\
&+c(p-1)u^{r+1}v^{p-2}
\left|Y-\theta\frac vuX\right|^2\\
&+b\,u^rv^{2p}-c\,u^{r+q+1}v^{p-1},
\end{aligned}
\end{equation}
where
\begin{equation}\label{eq:quartic-coefficient}
A(r,k)=-c_Nk^2+(r-1)k-\frac{r(r-1)}2.
\end{equation}
\end{lemma}

\begin{proof}
Because \(u,v>0\), they are bounded away from zero on compact
subsets. The nonlinearities are therefore smooth there, and local
elliptic regularity justifies the following differentiations.
Recall that
\[
\Delta u=-v^p,\qquad \Delta v=-u^q,
\qquad
\nabla\Delta u=-pv^{p-1}Y,
\qquad
\nabla|X|^2=2D^2u\,X.
\]

We first differentiate the four summands in
\eqref{eq:current-general}. The product rule gives
\begin{equation}\label{eq:four-divergences}
\begin{aligned}
\Div(u^rD^2u\,X)
={}&u^r|D^2u|^2
+r u^{r-1}D^2u[X,X]
-p u^rv^{p-1}X\cdot Y,\\
\Div(u^{r-1}|X|^2X)
={}&(r-1)u^{r-2}|X|^4
+2u^{r-1}D^2u[X,X]
-u^{r-1}v^p|X|^2,\\
\Div(u^rv^pX)
={}&r u^{r-1}v^p|X|^2
+p u^rv^{p-1}X\cdot Y
-u^rv^{2p},\\
\Div(u^{r+1}v^{p-1}Y)
={}&(r+1)u^rv^{p-1}X\cdot Y\\
&+(p-1)u^{r+1}v^{p-2}|Y|^2
-u^{r+q+1}v^{p-1}.
\end{aligned}
\end{equation}
The first equality uses
\[
\Div(D^2u\,X)=|D^2u|^2+X\cdot\nabla\Delta u,
\]
and the last equality uses \(\Delta v=-u^q\).

Multiplying the four lines of \eqref{eq:four-divergences} by the
coefficients in \eqref{eq:current-general} and collecting like
terms yields
\begin{equation}\label{eq:raw-divergence}
\begin{aligned}
\Div J={}&u^r|D^2u|^2
+2ku^{r-1}D^2u[X,X]\\
&+\left(k-\frac r2\right)(r-1)u^{r-2}|X|^4\\
&+\left(-k+\frac r2+\frac rN-rb\right)
u^{r-1}v^p|X|^2\\
&+\left[-p(c_N+b)+c(r+1)\right]
u^rv^{p-1}X\cdot Y\\
&+c(p-1)u^{r+1}v^{p-2}|Y|^2\\
&+\left(b-\frac1N\right)u^rv^{2p}
-c\,u^{r+q+1}v^{p-1}.
\end{aligned}
\end{equation}
The coefficient \(2k\) in the second term is the sum
\[
r+2\left(k-\frac r2\right)=2k.
\]

For the trace-free Hessian term, set
\[
\mathcal H=D^2u+\frac kuX\otimes X.
\]
Since
\[
|\mathcal H|^2
=|D^2u|^2+\frac{2k}{u}D^2u[X,X]
+\frac{k^2}{u^2}|X|^4
\]
and
\[
\tr\mathcal H
=-v^p+\frac{k}{u}|X|^2,
\]
the identity
\(
|\mathcal H^\circ|^2
=|\mathcal H|^2-(\tr\mathcal H)^2/N
\)
gives
\begin{equation}\label{eq:trace-expansion}
\begin{aligned}
u^r\left|\left(D^2u+\frac kuX\otimes X\right)^\circ\right|^2
={}&u^r|D^2u|^2+2ku^{r-1}D^2u[X,X]\\
&+c_Nk^2u^{r-2}|X|^4
+\frac{2k}{N}u^{r-1}v^p|X|^2-\frac1Nu^rv^{2p}.
\end{aligned}
\end{equation}

Subtract the right-hand side of \eqref{eq:trace-expansion} from
the expansion \eqref{eq:raw-divergence}, the coefficient of
\(u^{r-2}|X|^4\) becomes
\[
\left(k-\frac r2\right)(r-1)-c_Nk^2
=-c_Nk^2+(r-1)k-\frac{r(r-1)}2
=A(r,k),
\]
then the two \(u^rv^{2p}/N\) terms in $\ref{eq:raw-divergence}$ and $\ref{eq:trace-expansion}$ combine to leave
\(b\,u^rv^{2p}\). 

Hence
\[
\begin{aligned}
\Div J
={}&u^r|\mathcal H^\circ|^2
+A(r,k)u^{r-2}|X|^4\\
&+b\,u^rv^{2p}
-c\,u^{r+q+1}v^{p-1}
+\mathcal Q,
\end{aligned}
\]
where the remaining quadratic expression is
\begin{equation}\label{eq:quadratic-before-square}
\begin{aligned}
 \mathcal Q={}&
[\lambda-r(c_N+b)]u^{r-1}v^p|X|^2\\
&+[-p(c_N+b)+c(r+1)]u^rv^{p-1}X\cdot Y\\
&\quad
+c(p-1)u^{r+1}v^{p-2}|Y|^2.
\end{aligned}
\end{equation}

Expanding the gradient square gives

\begin{equation}\label{eq:10-061}
\begin{aligned}
&c(p-1)u^{r+1}v^{p-2}
\left|Y-\theta\frac vuX\right|^2\\
&\quad=
c(p-1)u^{r+1}v^{p-2}|Y|^2
-2c(p-1)\theta\,u^rv^{p-1}X\cdot Y\\
&\qquad\quad
+c(p-1)\theta^2u^{r-1}v^p|X|^2.
\end{aligned}
\end{equation}

Next, we prove under the assumption of lemma $\ref{lem:identity}$, $\mathcal Q$ is identity to the right hand of $\ref{eq:10-061}$.

The coefficient of \(|Y|^2\)  already agrees with
\eqref{eq:quadratic-before-square}. Equalities of the coefficients of mixed and
\(|X|^2\)  are therefore respectively equivalent to
\begin{equation}\label{eq:matching}
p(c_N+b)=c\tau,\qquad
\lambda-r(c_N+b)=c(p-1)\theta^2.
\end{equation}
The first equation in \eqref{eq:matching} is equivalent to
\[
-p(c_N+b)+c(r+1)=-2c(p-1)\theta.
\]
It gives \(c_N+b=c\tau/p\). Substituting this into the second
equation and multiplying by \(p\), we obtain
\[
p\lambda
=c\bigl[r\tau+p(p-1)\theta^2\bigr]
=c\delta.
\]
Since \(\delta\ne0\),
\[
c=\frac{\lambda p}{\delta},
\qquad
b=\frac{c\tau}{p}-c_N
=\frac{\lambda\tau}{\delta}-c_N.
\]
With these values of $b,c$, the expression $\mathcal Q$ is the
required square. Substituting it into the expression for $\Div J$, we
prove \eqref{eq:identity}.
\end{proof}

In fact, the discussion above is for $J_u$, here $b,c,\delta,\lambda,\tau$ in lemma \ref{lem:identity} is $b_u,c_u,\delta_u,\lambda_u,\tau_u$ respectively. For $J_v$, we replace
$(r,p,u,v)$ by $(s,q,v,u)$ and use the similar approach to get the similar identity and $b_v, c_v,\delta_v,\lambda_v,\tau_v$.

For the two components, use subscripts $u,v$, we write
\[
A_u=A(r,k_u),\quad A_v=A(s,k_v),\qquad
Q_u=u^{r-2}|\nabla u|^4,\quad Q_v=v^{s-2}|\nabla v|^4.
\]

\subsubsection{Combining the two component identities}

We match the scaling degrees of $u^rv^{2p}, v^su^{2q}$ before
combining the two vector fields. The criterion below ensures
that the combined divergence controls $u^rv^{2p}, v^su^{2q}$
and the quartic gradient terms.

Recall that 
\begin{equation}\label{eq:balanced}
\tilde{\kappa}=\frac{N-3}{N},\qquad
r+2=\tilde{\kappa}(q+1),\qquad s+2=\tilde{\kappa}(p+1).
\end{equation}
Set
\begin{equation}\label{eq:energies}
F_1=u^rv^{2p},\qquad F_2=u^{2q}v^s.
\end{equation}
The relations \eqref{eq:balanced} give $F_1$ and $F_2$ the same
homogeneity under the scaling of positive solution $(u,v)$.

Indeed, the scaling degrees of $F_1$ and $F_2$ are
$r\alpha+2p\beta$ and $2q\alpha+s\beta$, respectively.
Using $p\beta=\alpha+2$, $q\alpha=\beta+2$ and
\eqref{eq:balanced}, we have
\[
\begin{aligned}
r\alpha+2p\beta
&=(r+2)\alpha+4
=\tilde{\kappa}(q+1)\alpha+4\\
&=\tilde{\kappa}(p+1)\beta+4
=(s+2)\beta+4
=2q\alpha+s\beta.
\end{aligned}
\]

The following proposition requires only \eqref{eq:balanced} and the 
coefficient inequalities assumptions $\ref{eq:target-signs}$. We verify these inequalities
for $(p,q)\in\mathcal R_N$ in Subsection~\ref{sec:parameter-positivity}.

\begin{proposition}\label{prop:criterion}
Let $p,q>1$, the weights satisfying \eqref{eq:balanced}.
If
\begin{equation}\label{eq:target-signs}
A_u,A_v,b_u,b_v,c_u,c_v>0,\qquad b_ub_v>c_uc_v,
\end{equation}
then there exist constants $\omega,\varepsilon_0>0$ such that
\begin{equation}\label{eq:coercivity}
\Div(J_u+\omega J_v)
\ge\varepsilon_0(Q_u+Q_v+F_1+F_2).
\end{equation}
Here the constants $\omega, \varepsilon_0 $ depend on the coefficients $\eqref{eq:target-signs}$, but not on the solution  $(u,v)$.
\end{proposition}

\begin{proof}
For any $\omega>0$, Lemma~\ref{lem:identity} gives
\[
\begin{aligned}
\Div(J_u+\omega J_v)
&=\Div J_u+\omega\Div J_v\\
&\ge A_uQ_u+\omega A_vQ_v
     +b_uF_1+\omega b_vF_2\\
&\quad-c_u u^{r+q+1}v^{p-1}
      -\omega c_v u^{q-1}v^{p+s+1}.
\end{aligned}
\]
Here we have discarded the nonnegative square terms,
using $p,q>1$ and $c_u,c_v>0$.

Put
\[
\rho=\frac{v^{p+1}}{u^{q+1}}>0,
\qquad \sigma=1-\tilde{\kappa}>0.
\]
By \eqref{eq:balanced}, we have
\[
\begin{aligned}
\frac{F_1}{F_2}
&=u^{r-2q}v^{2p-s}
=\rho^{1+\sigma},\\
\frac{u^{r+q+1}v^{p-1}}{F_2}
&=u^{r-q+1}v^{p-1-s}
=\rho^\sigma,\\
\frac{u^{q-1}v^{p+s+1}}{F_2}
&=u^{-q-1}v^{p+1}
=\rho.
\end{aligned}
\]
Consequently,
\begin{equation}\label{eq:10062}
\begin{aligned}
\Div(J_u+\omega J_v)
&\ge A_uQ_u+\omega A_vQ_v\\
&\quad+F_2\bigl[
b_u\rho^{1+\sigma}-c_u\rho^\sigma
+\omega b_v-\omega c_v\rho
\bigr].
\end{aligned}
\end{equation}

Since $b_ub_v>c_uc_v$, we have
\[
a_-:=\frac{c_u}{b_u}
<
a_+:=\frac{b_v}{c_v}.
\]
Choose
\[
m=\frac{a_-+a_+}{2},
\qquad
\omega=\frac{b_um^\sigma}{c_v}>0.
\]
Then
\[
\begin{aligned}
&b_u\rho^{1+\sigma}-c_u\rho^\sigma
+\omega b_v-\omega c_v\rho\\
&\qquad
=b_u\bigl[\rho^\sigma(\rho-a_-)
          +m^\sigma(a_+-\rho)\bigr]
=:b_u\Phi(\rho).
\end{aligned}
\]
We rewrite $\Phi$ as
\begin{equation}\label{eq:Phi}
\begin{aligned}
\Phi(\rho)
&=\rho^\sigma(\rho-a_-)+m^\sigma(a_+-\rho)\\
&=(\rho^\sigma-m^\sigma)(\rho-m)
 +(m-a_-)\rho^\sigma+(a_+-m)m^\sigma.
\end{aligned}
\end{equation}
The first term is nonnegative because $\sigma>0$,
and the other two are positive because $a_-<m<a_+$.
Thus $\Phi(\rho)>0$ for every $\rho>0$.

Moreover,
\[
\lim_{\rho\downarrow0}
\frac{\Phi(\rho)}{1+\rho^{1+\sigma}}
=m^\sigma a_+>0,
\qquad
\lim_{\rho\to\infty}
\frac{\Phi(\rho)}{1+\rho^{1+\sigma}}=1.
\]
Positivity and continuity therefore give
\[
c_0:=\inf_{\rho>0}
\frac{\Phi(\rho)}{1+\rho^{1+\sigma}}>0.
\]
Set
\[
\varepsilon_0
=\min\{A_u,\omega A_v,b_uc_0\}>0.
\]
Returning to the divergence estimate $\eqref{eq:10062}$, we obtain
\[
\begin{aligned}
\Div(J_u+\omega J_v)
&\ge A_uQ_u+\omega A_vQ_v+b_uF_2\Phi(\rho)\\
&\ge A_uQ_u+\omega A_vQ_v
     +b_uc_0F_2(1+\rho^{1+\sigma})\\
&=A_uQ_u+\omega A_vQ_v+b_uc_0(F_1+F_2)\\
&\ge\varepsilon_0(Q_u+Q_v+F_1+F_2).
\end{aligned}
\]
The constants $\omega$, $\varepsilon_0$ are independent of the solution.
\end{proof}

\subsection{Parameter choice and coefficient positivity}\label{sec:parameter-positivity}

In this subsection, we assume $(p,q)\in\mathcal R_N$.

\subsubsection{Choice of parameters}\label{sec:parameters}

We choose the parameters and prove $-2<r,s<0$ and
$A_u,A_v>0$ and the numerators of $b_u,b_v,c_u,c_v$  in \eqref{eq:target-signs}
are positive. The next subsection proves
$\delta_u,\delta_v>0$ and $b_ub_v>c_uc_v$.

Fix $\tilde{\kappa},r,s$ as in \eqref{eq:chosen-weights}, these weights
satisfy \eqref{eq:balanced}. Write
\[
x_u=\tilde{\kappa}(q-P)<0,\qquad x_v=\tilde{\kappa}(p-P)>0.
\]
Thus
\[
r=-\frac2{N-2}+x_u,\qquad
s=-\frac2{N-2}+x_v.
\]
For $i=u,v$, choose
\begin{equation}\label{eq:chosen}
\begin{aligned}
k_i&=-\frac N{N-2}+\frac{x_i}{2}+\frac{x_i^2}{4},&
\mu_i&=\frac N{N-2}+x_i-\frac{N+2}{4(N-1)}x_i^2,\\
\theta_u&=\frac{\mu_u-r}{p},&
\theta_v&=\frac{\mu_v-s}{q}.
\end{aligned}
\end{equation}
Indeed, if $r=-2/(N-2)+x$ and
$k=-N/(N-2)+a_1x+a_2x^2$ are substituted in $A(r,k)$,
the constant term vanishes and the linear term is
$(a_1-1/2)x$. Hence $a_1=1/2$ removes the term that changes sign
between the two components. Taking $a_2=1/4$ gives
\begin{equation}\label{eq:chosen-identities}
A_i=\frac{x_i^2[4+4x_i-(N-1)x_i^2]}{16N},
\qquad i=u,v.
\end{equation}

Notice $c_N=(N-1)/N$, the choices \eqref{eq:chosen} satisfy
\[
\lambda_u=c_N\mu_u,\qquad \lambda_v=c_N\mu_v,
\]
where $\lambda_u,\lambda_v$ are the coefficients from
Lemma~\ref{lem:identity}. Substituting these relations and
$\theta_u=(\mu_u-r)/p$, $\theta_v=(\mu_v-s)/q$ gives the following
coefficient expressions whenever the denominators are nonzero:
\begin{equation}\label{eq:coefficient-definitions}
\begin{aligned}
\delta_u&=r+\frac{r^2+(p-1)\mu_u^2}{p},&
\delta_v&=s+\frac{s^2+(q-1)\mu_v^2}{q},\\
c_u&=\frac{c_Np\mu_u}{\delta_u},&
c_v&=\frac{c_Nq\mu_v}{\delta_v},\\
b_u&=\frac{c_N(\mu_u-r)[p+r+(p-1)\mu_u]}{p\delta_u},&
b_v&=\frac{c_N(\mu_v-s)[q+s+(q-1)\mu_v]}{q\delta_v}.
\end{aligned}
\end{equation}

All denominators will be proved positive in
Section~\ref{sec:coefficient-positivity}.

\proofstep{The normalized variables.}

Use the variables
\[
h=\sqrt P-1,\qquad
p=P+hy,\qquad q=P-hy-h^2z.
\]
Then
\begin{equation}\label{eq:h-map}
P=(1+h)^2,\qquad
N=\frac{2(P+1)}{h(h+2)},\qquad
\tilde{\kappa}=\frac{5-P}{2(P+1)}.
\end{equation}
For $(p,q)  \in$ $\mathcal R_N$, we have
\begin{equation}\label{eq:box}
0<h<\frac8{15},\qquad 0<y,z\le1.
\end{equation}
Indeed, the upper circular boundary is decreasing for
$q\ge\sqrt P$ and has $p=P+h$ at $q=\sqrt P$, so $0<y\le1$.
Moreover, $p+q<2P$ gives $z>0$, while
$q(p-P+1)\ge P$ becomes
\[
y(h+2-y)-z(1+hy)\ge0.
\]
Since
\[
1+hy-y(h+2-y)=(1-y)^2\ge0,
\]
we obtain
\[
z\le\frac{y(h+2-y)}{1+hy}\le1.
\]
The disk condition is
\begin{equation}\label{eq:g}
G_N=h^2g,\qquad
g=(P+1)z-2y^2-2hyz-h^2z^2\ge0.
\end{equation}
Finally,
\begin{equation}\label{eq:transformed-basic-quantities}
\begin{gathered}
x_u=-\tilde{\kappa} h(y+hz),\qquad x_v=\tilde{\kappa} hy,\\
r=h\bar r,\qquad s=h\bar s,\\
\bar r=-\frac{h+2}{2}-\tilde{\kappa}(y+hz),\qquad
\bar s=-\frac{h+2}{2}+\tilde{\kappa} y,\\
p-1=h(2+h+y),\qquad
q-1=h(2+h-y-hz).
\end{gathered}
\end{equation}

\proofstep{The weights.}

From \eqref{eq:chosen-weights}, $r,s>-2$. Formula
\eqref{eq:transformed-basic-quantities} gives $r<0$ and
\[
s\le-\frac{h(h+2)}2+\tilde{\kappa} h<0,
\]
because $0<\tilde{\kappa}<1$. Hence $-2<r,s<0$.

\proofstep{The quartic coefficients.}

From the definition \eqref{eq:chosen-identities}, multiplication by
the positive factor $16N/x_i^2$ gives
\begin{equation}\label{eq:A-normalization}
\frac{16N}{x_i^2}A_i=4+4x_i-(N-1)x_i^2.
\end{equation}
Set $a=\tilde{\kappa} h(1+h)$. The parameter range gives
\[
-a\le x_u<0<x_v\le\tilde{\kappa} h\le a.
\]
Furthermore,
\begin{equation}\label{eq:direct-half-bound}
\begin{aligned}
2(P+1)\left(\frac12-a\right)
&=2-2h-h^2+3h^3+h^4\\
&\ge2-\frac{16}{15}-\frac{64}{225}
=\frac{146}{225}>0,
\end{aligned}
\end{equation}
so $a<1/2$. Since $h\le2/(N-2)$,
\begin{equation}\label{eq:direct-N-bound}
(N-1)\tilde{\kappa} h
\le\frac{2(N-1)(N-3)}{N(N-2)}<2.
\end{equation}
Since $a<1/2$, \eqref{eq:direct-N-bound} gives
\[
4a<2,\qquad
(N-1)a^2=[(N-1)\tilde{\kappa} h]a(1+h)<1+h.
\]
Consequently,
\[
4+4x_i-(N-1)x_i^2
\ge4-4a-(N-1)a^2>1-h>0.
\]
Equation~\eqref{eq:A-normalization} proves $A_u,A_v>0$.

\proofstep{The numerator factors.}

For the numerator factors in \eqref{eq:coefficient-definitions},
the bounds
$(N+2)/(4(N-1))<1/2$ and $|x_i|\le a<1/2$ give
\begin{equation}\label{eq:direct-mu-bound}
\mu_i\ge\frac N{N-2}-a-\frac{N+2}{4(N-1)}a^2
>1-\frac12-\frac18=\frac38.
\end{equation}
Hence $\mu_u-r,\mu_v-s>0$. Also,
\begin{equation}\label{eq:direct-numerator-bounds}
p+r>\frac N{N-2}-a>\frac12,\qquad
q+s>\sqrt P-\frac2{N-2}=1-\frac{h^2}{2}>0.
\end{equation}

Hence the numerators of $b_u,b_v,c_u,c_v$ in \eqref{eq:coefficient-definitions} are positive.

\subsubsection{Positivity of the remaining coefficients}
\label{sec:coefficient-positivity}\label{sec:certificate-definitions}

To prove $\delta_u,\delta_v>0$ and $b_ub_v>c_uc_v$, we reduce
the problem to four polynomial inequalities.
The disk condition is used to prove the product inequality.

\proofstep{The denominator polynomials.}

Put
$L=4(P+1)^2(P+3)>0$ and define the two normalized denominators
\begin{equation}\label{eq:delta-hat}
\widehat\delta_u:=\frac{pL^2}{h}\delta_u,
\qquad
\widehat\delta_v:=\frac{qL^2}{h}\delta_v.
\end{equation}
Since $p,q,L,h>0$, the two multipliers are positive. Thus it
suffices to prove $\widehat\delta_u,\widehat\delta_v>0$.
For the first denominator, substitute its definition in
\eqref{eq:coefficient-definitions}, then use
$r=h\bar r$ and $p-1=h(2+h+y)$ from
\eqref{eq:transformed-basic-quantities}:
\[
\begin{aligned}
\widehat\delta_u
&=\frac{L^2}{h}\bigl[pr+r^2+(p-1)\mu_u^2\bigr]\\
&=\frac{L^2}{h}\bigl[ph\bar r+h^2\bar r^2
       +h(2+h+y)\mu_u^2\bigr]\\
&=L^2\bigl[\bar r(p+h\bar r)+(2+h+y)\mu_u^2\bigr].
\end{aligned}
\]
Similarly, $s=h\bar s$ and $q-1=h(2+h-y-hz)$ give
\[
\begin{aligned}
\widehat\delta_v
&=\frac{L^2}{h}\bigl[qs+s^2+(q-1)\mu_v^2\bigr]\\
&=\frac{L^2}{h}\bigl[qh\bar s+h^2\bar s^2
       +h(2+h-y-hz)\mu_v^2\bigr]\\
&=L^2\bigl[\bar s(q+h\bar s)+(2+h-y-hz)\mu_v^2\bigr].
\end{aligned}
\]

%The factors $P+1$, $P+3$, and $5-P$ occur repeatedly.
We write 
\[
B=P+1,\qquad C=P+3,\qquad K=5-P,
\]
thus $L=4B^2C$ and $\tilde{\kappa}=K/(2B)$. 

Define two polynomials in
$h$ and $w$:
\begin{equation}\label{eq:polynomial-building-blocks}
\begin{aligned}
E(w)&=2BC\bigl[Kw-B(h+2)\bigr],\\
M(w)&=2B^3C+2BCKhw-PK^2h^2w^2.
\end{aligned}
\end{equation}
Here $P,B,C,K$ are themselves polynomials in $h$. Set
\[
E_u=E(-y-hz),\quad E_v=E(y),\qquad
M_u=M(-y-hz),\quad M_v=M(y).
\]
These four expressions are therefore polynomials in $h,y,z$.

Take $w=-y-hz$ for the $u$-component and $w=y$ for the $v$-component,
by \eqref{eq:transformed-basic-quantities} and \eqref{eq:chosen},
 $\bar r,\mu_u$ in the first $u$ case and $\bar s,\mu_v$ in the second $v$
are the two expressions
\[
-\frac{h+2}{2}+\frac{Kw}{2B},
\qquad
\frac B2+\frac{Khw}{2B}-\frac{PK^2h^2w^2}{4B^2C}.
\]

Multiplying them by $L=4B^2C$ cancels the denominators and
gives $E(w),M(w)$. Hence
\begin{equation}\label{eq:polynomial-component-relations}
E_u=L\bar r,\qquad E_v=L\bar s,\qquad
M_u=L\mu_u,\qquad M_v=L\mu_v.
\end{equation}

Substitution into the two formulas above gives
\begin{equation}\label{eq:delta-hat-polynomial}
\begin{aligned}
\widehat\delta_u
&=E_u(pL+hE_u)+(2+h+y)M_u^2,\\
\widehat\delta_v
&=E_v(qL+hE_v)+(2+h-y-hz)M_v^2.
\end{aligned}
\end{equation}

These expressions contain no division by $h$, so they also
define the normalized denominators at $h=0$. All four polynomial
factors $E_u,E_v,M_u,M_v$ have been fixed by the parameter choices.

\proofstep{The product polynomial.}

Using the same polynomial factors, define
\begin{equation}\label{eq:product-polynomial-parts}
\begin{aligned}
\mathcal U={}&(M_u-hE_u)(M_v-hE_v)
             [pL+hE_u+(p-1)M_u]\\
            &\hspace{22mm}\times[qL+hE_v+(q-1)M_v],\\
\mathcal V={}&p^2q^2L^2M_uM_v.
\end{aligned}
\end{equation}
Once $\delta_u,\delta_v>0$ has been established,
\eqref{eq:Q-hat-expansion}--\eqref{eq:Q-hat}
show that $b_ub_v-c_uc_v$ has the same sign as
$\mathcal{U}-\mathcal{V}$.
It therefore suffices to prove $\mathcal{U}-\mathcal{V}>0$.

To verify that $h^3$ divides this difference, expand the two
products in powers of $h$. Their first four coefficients, after
division by $64^4=2^{24}$, are recorded in
Table~\ref{tab:low-order-products}. They follow by multiplying the
explicit factors in \eqref{eq:polynomial-building-blocks} and
\eqref{eq:product-polynomial-parts}, retaining powers of $h$ through
degree three.

\begin{table}[htbp]
\centering\small
\caption{The coefficient of $h^j$ in the two normalized products.}
\label{tab:low-order-products}
\begin{tabular}{@{}ccc@{}}
\toprule
$j$ & $\mathcal U/2^{24}$ & $\mathcal V/2^{24}$\\
\midrule
$0$ & $1$ & $1$\\
$1$ & $20$ & $20$\\[1mm]
$2$ & $\dfrac{393-7y^2-6z}{2}$
    & $\dfrac{393-7y^2-6z}{2}$\\[3mm]
$3$ & $\dfrac{5056-221y^2-14yz-210z}{4}$
    & $\dfrac{5056-223y^2-14yz-214z}{4}$\\[2mm]
\bottomrule
\end{tabular}
\end{table}

The constant, linear, and quadratic coefficients are identical.
They cancel in $\mathcal U-\mathcal V$, so every remaining monomial
contains at least $h^3$. We may therefore define the polynomial
\begin{equation}\label{eq:Q-divisibility}
\widehat Q=\frac{\mathcal U-\mathcal V}{h^3}.
\end{equation}
Division here simply lowers the exponent of $h$ in each remaining
monomial by three. The last row of the table also gives
\[
\widehat Q(0,y,z)=2^{23}(y^2+2z).
\]

\proofstep{Using the disk condition.}

We next construct the factors needed to use $g\ge0$ in \eqref{eq:g}.
At $z=0$, replacing $y$ by $-y$ interchanges $p,q$, $E_u,E_v$,
and $M_u,M_v$. Both products in
\eqref{eq:product-polynomial-parts} are unchanged, so
$\widehat Q(h,y,0)$ contains only even powers of $y$.
Moreover, at $y=z=0$,
\[
p=q=P,\qquad M_u=M_v=\frac{BL}{2},\qquad
hE_u=hE_v=-\frac{(P-1)L}{2}.
\]
Consequently,
\[
\mathcal U=(PL)^2\left(\frac{PBL}{2}\right)^2
=P^4L^2\left(\frac{BL}{2}\right)^2=\mathcal V,
\]
and hence $\widehat Q(h,0,0)=0$. The even polynomial
$\widehat Q(h,y,0)$ therefore has no constant term in $y$,
which proves that $y^2$ divides it.

Write the finite expansion of the polynomial already fixed by
\eqref{eq:Q-divisibility}:
\[
\widehat Q(h,y,z)=\sum_{i,j,k\ge0}a_{ijk}h^iy^jz^k.
\]
Its coefficients $a_{ijk}$ are determined by
\eqref{eq:polynomial-building-blocks}--\eqref{eq:product-polynomial-parts}.
The preceding argument gives $a_{i00}=a_{i10}=0$ for every $i$.
Define
\[
T_*=\sum_{\substack{i\ge0\\j\ge2}}a_{ij0}h^iy^{j-2},
\qquad
R_*=\sum_{\substack{i,j\ge0\\k\ge1}}a_{ijk}h^iy^jz^{k-1}.
\]
All exponents are nonnegative, so these are polynomials. Grouping
the terms with $k=0$ and $k\ge1$ gives
\begin{equation}\label{eq:Q-first}
\widehat Q=y^2T_*+zR_*.
\end{equation}
In particular, $T_*=\widehat Q(h,y,0)/y^2$ and
$R_*=[\widehat Q(h,y,z)-\widehat Q(h,y,0)]/z$ are exact polynomial
quotients, with their values on $y=0$ or $z=0$ fixed by the finite
sums above.

Finally, the disk identity \eqref{eq:g} reads
\[
2y^2=(B-2hy-h^2z)z-g.
\]
Substituting this into twice \eqref{eq:Q-first} gives
\begin{equation}\label{eq:Q-second}
2\widehat Q=-T_*g+zH_*,
\qquad H_*=2R_*+(B-2hy-h^2z)T_*.
\end{equation}
This also defines $H_*$ as a polynomial.

If $R_*,H_*>0$, then $\widehat Q>0$ for $h,y,z>0$ and $g\ge0$. Indeed
 we can use \eqref{eq:Q-first} when $T_*\ge0$ and \eqref{eq:Q-second}
when $T_*<0$. It therefore remains to check the signs of
$\widehat\delta_u,\widehat\delta_v,R_*,H_*$.

\proofstep{The Bernstein positivity criterion.}

Set $\mathcal B=[0,8/15]\times[0,1]^2$. The four polynomials above
have fixed rational coefficients in $h,y,z$. The dimension enters
through $h=\sqrt{(N+2)/(N-2)}-1$, and \eqref{eq:box} places every
parameter triple arising from $\mathcal R_N$ in $\mathcal B$.
Proving the four inequalities on this box therefore covers
all $N\ge5$. We use the following elementary criterion.
For $d\ge0$, set
\[
B_{i,d}(t)=\binom di t^i(1-t)^{d-i},\qquad i=0,\ldots,d.
\]

\begin{lemma}[Bernstein positivity]\label{lem:bernstein}
Let $f(h,y,z)$ have degree at most $d_1,d_2,d_3$ in the respective
variables. It has a unique representation
\begin{equation}\label{eq:bernstein-basis}
f\left(\frac8{15}\xi,y,z\right)
=\sum_{I=0}^{d_1}\sum_{J=0}^{d_2}\sum_{K=0}^{d_3}
\beta_{IJK}B_{I,d_1}(\xi)B_{J,d_2}(y)B_{K,d_3}(z),
\end{equation}

and satisfies
\begin{equation}\label{eq:bernstein-bound}
f(h,y,z)\ge\min_{I,J,K}\beta_{IJK}
\qquad\text{for all }(h,y,z)\in\mathcal B.
\end{equation}

In particular, positive Bernstein coefficients give a uniform
positive lower bound for $f$ on $\mathcal B$.
\end{lemma}

\begin{proof}
The lowest nonzero powers of $B_{0,d},\ldots,B_{d,d}$ are
$t^0,\ldots,t^d$, respectively. Comparing these powers proves
linear independence, so these polynomials form a basis. Taking
products of the one-variable bases and setting $h=(8/15)\xi$
gives \eqref{eq:bernstein-basis}. For $0\le t\le1$, the basis
functions are nonnegative and
\[
\sum_{i=0}^d B_{i,d}(t)=(t+(1-t))^d=1.
\]
Thus \eqref{eq:bernstein-basis} is a convex combination of its
coefficients on the unit cube, which proves
\eqref{eq:bernstein-bound}.
\end{proof}

\proofstep{Computing the exact Bernstein coefficients.}

Fix one of the four polynomials
$\widehat\delta_u,\widehat\delta_v,R_*,H_*$ from
this subsection, and write its ordinary
power expansion as
\[
f(h,y,z)=\sum_{i,j,k}c_{ijk}h^iy^jz^k.
\]
The numbers $c_{ijk}$ are the coefficients of the formulas in
this subsection, with all parameters
already fixed in Section~\ref{sec:parameters}.

Each polynomial belongs to $\mathbb Q[h,y,z]$ and is constructed
by finitely many additions, multiplications, substitutions, and
the exact divisions by $h^3$, $y^2$, or $z$ justified above.
Performing these operations in rational arithmetic determines
every $c_{ijk}$ exactly.

For fixed coordinate degrees $(d_1,d_2,d_3)$, the Bernstein
coefficients are then given by
\begin{equation}\label{eq:bernstein-coefficients}
\beta_{IJK}=
\sum_{\substack{i\le I,\ j\le J,\ k\le K}}
c_{ijk}\left(\frac8{15}\right)^i
\frac{\binom Ii}{\binom{d_1}i}
\frac{\binom Jj}{\binom{d_2}j}
\frac{\binom Kk}{\binom{d_3}k}.
\end{equation}
To derive this formula, apply the binomial identity
\[
t^j=\sum_{i=j}^d
\frac{\binom ij}{\binom dj}B_{i,d}(t)
\]
to every monomial in each variable. The factor $(8/15)^i$
comes from the change $h=(8/15)\xi$. Every sum in
\eqref{eq:bernstein-coefficients} is finite and every factor is
rational. Thus all Bernstein coefficients can be computed exactly.
The complete table has
$(d_1+1)(d_2+1)(d_3+1)$ entries.

\begin{lemma}\label{lem:polynomial-bounds}
The four polynomials defined above satisfy
\[
\widehat\delta_u,\widehat\delta_v\ge4096,\qquad
R_*\ge2^{24},\qquad H_*\ge3\cdot2^{24}
\quad\text{on }\mathcal B.
\]
\end{lemma}

\begin{proof}
The program \texttt{verify.py} constructs the four polynomials
from the formulas above and computes all their Bernstein
coefficients by \eqref{eq:bernstein-coefficients}, using exact
rational arithmetic. After substituting $h=(8/15)\xi$, it applies
the one-variable change of basis successively in $\xi,y,z$.
For each polynomial, it reports the coordinate degrees, the
number of coefficients and the smallest coefficient. These
results are recorded in Appendix~\ref{app:data}.
Lemma~\ref{lem:bernstein} turns these coefficient minima into the
stated lower bounds on $\mathcal B$.
\end{proof}

For the parameters under consideration, $h,y,z>0$ and $g\ge0$.
If $T_*\ge0$, \eqref{eq:Q-first} gives $\widehat Q>0$ since
$zR_*>0$. If $T_*<0$, \eqref{eq:Q-second} gives
$2\widehat Q\ge zH_*>0$. Thus the normalized denominators
and the product polynomial have the required signs.

\proofstep{Returning to the coefficients.}

Since $p,q,L,h>0$, \eqref{eq:delta-hat} gives
$\delta_u,\delta_v>0$. The positive numerator factors in
\eqref{eq:direct-mu-bound}--\eqref{eq:direct-numerator-bounds}
and the formulas \eqref{eq:coefficient-definitions} therefore give
$b_u,b_v,c_u,c_v>0$.

To verify the product inequality, first multiply the  $b_u,b_v$
coefficients:
\[
\begin{aligned}
b_ub_v={}&\frac{c_N^2}{pq\delta_u\delta_v}
(\mu_u-r)(\mu_v-s)[p+r+(p-1)\mu_u]\\
&\hspace{32mm}\times[q+s+(q-1)\mu_v].
\end{aligned}
\]
The product of the  $c_u,c_v$ coefficients is
\[
c_uc_v=\frac{c_N^2pq\mu_u\mu_v}{\delta_u\delta_v}
      =\frac{c_N^2p^2q^2\mu_u\mu_v}{pq\delta_u\delta_v}.
\]

Both products now have the common denominator
$pq\delta_u\delta_v$. Subtract them and multiply by
$pqL^4\delta_u\delta_v/c_N^2$. By
\eqref{eq:polynomial-component-relations}, the result is
$\mathcal U-\mathcal V=h^3\widehat Q$, or equivalently,
\begin{equation}\label{eq:Q-hat-expansion}
\begin{aligned}
h^3\widehat Q=L^4\bigl\{&
(\mu_u-r)(\mu_v-s)
[p+r+(p-1)\mu_u]\\
&\times[q+s+(q-1)\mu_v]-p^2q^2\mu_u\mu_v\bigr\}.
\end{aligned}
\end{equation}

Consequently,
\begin{equation}\label{eq:Q-hat}
\widehat Q=
\frac{pqL^4\delta_u\delta_v}{c_N^2h^3}
\bigl(b_ub_v-c_uc_v\bigr).
\end{equation}

The prefactor is now known to be positive, and $\widehat Q>0$.
Hence $b_ub_v-c_uc_v>0$, completing the coefficient inequalities.

\begin{proof}[\normalfont\bfseries Proof of Proposition~\ref{prop:circle-coefficients}]
The estimates above give $-2<r,s<0$ and all the signs in
\eqref{eq:target-signs}. In particular, the denominators in
Lemma~\ref{lem:identity} are nonzero. Define $J_u,J_v$ by that
lemma with the chosen parameters. Proposition~\ref{prop:criterion}
then gives \eqref{eq:constructed-coercivity}, with constants
depending only on $N,p,q$.
\end{proof}

\begin{remark}
The argument includes $g=0$, hence the relevant circular boundary.
At $(p,q)=(P,P)$, however, $x_u=x_v=0$ and
\eqref{eq:chosen-identities} gives $A_u=A_v=0$. Hence the above proof does not hold for $(P, P)$.
\end{remark}

\section{Local estimates and proof of Theorem~\ref{thm:main}}\label{sec:local}

By Proposition~\ref{prop:reduction}, it remains to prove
nonexistence in $\mathcal R_N$. We use the divergence estimate
proved in Section~\ref{sec:currents} to obtain a local integral
bound independent of the solution, then conclude by scaling.

Assume \((p,q)\in\mathcal R_N\), and take the vector fields given
by Proposition~\ref{prop:circle-coefficients}. The local estimate
we need is the following.

\begin{proposition}\label{prop:local-bound}
For fixed \(N,p,q\) with \((p,q)\in\mathcal R_N\), every positive
entire solution $(u,v)$ of \eqref{eq:system} satisfies
\begin{equation}\label{eq:local-bound}
\int_{B_1}(u^rv^{2p}+u^{2q}v^s)\dd\le C(N,p,q).
\end{equation}
\end{proposition}

In this section, all balls are centered at the origin unless otherwise indicated.

\subsection{A local \texorpdfstring{$L^1$}{L1} bound}

The interpolation argument requires a lower-order integral bound
independent of the solution. Testing the system with the first
Dirichlet eigenfunction on a ball gives local $L^1$ control of
both components.

Let $\psi$ be the positive first Dirichlet eigenfunction on $B_4$,
normalized by $\int_{B_4}\psi\dd=1$, and let $\lambda_1$ be its
eigenvalue. Here $\nu$ is the outward unit normal.
Hopf's lemma gives $\partial_\nu\psi<0$ on $\partial B_4$.
Green's formula then gives
\[
\int_{B_4}v^p\psi\dd
=\lambda_1\int_{B_4}u\psi\dd
+\int_{\partial B_4}u\,\partial_\nu\psi\,\mathrm dS
\le\lambda_1\int_{B_4}u\psi\dd.
\]
Applying the same argument to \(v\) gives
\[
\int_{B_4}u^q\psi\dd
=\lambda_1\int_{B_4}v\psi\dd
+\int_{\partial B_4}v\,\partial_\nu\psi\,\mathrm dS
\le\lambda_1\int_{B_4}v\psi\dd.
\]

Set $U=\int_{B_4}u\psi\dd$ and $V=\int_{B_4}v\psi\dd$.
Since $p,q>1$, Jensen's inequality gives
\[
V^p\le\lambda_1U,\qquad U^q\le\lambda_1V.
\]
Since $pq>1$, these inequalities imply
\[
U\le\lambda_1^{(p+1)/(pq-1)},\qquad
V\le\lambda_1^{(q+1)/(pq-1)}.
\]
The positive minimum of $\psi$ on $\overline B_2$ therefore yields
\begin{equation}\label{eq:L1}
\int_{B_2}(u+v)\dd\le C(N,p,q).
\end{equation}

\subsection{The cutoff estimate}

We derive the cutoff estimate needed for the bound on $B_1$.
Integration by parts and absorption of the cutoff errors leave
lower powers of the two components, which are estimated in the
next subsection.

To remove the Hessian from the boundary terms, use
$D^2u\nabla u=\frac12\nabla|\nabla u|^2$ and the product rule
to write
\begin{equation}\label{eq:weak-form}
\begin{aligned}
J_u={}&\nabla\left(\frac12u^r|\nabla u|^2
+\frac{c_u}{p}u^{r+1}v^p\right)\\
&-(r-k_u)u^{r-1}|\nabla u|^2\nabla u\\
&-\left(b_u-\frac1N+\frac{c_u(r+1)}p\right)
u^rv^p\nabla u.
\end{aligned}
\end{equation}
For $1\le \Gamma<\Gamma'\le2$, put $\ell=\Gamma'-\Gamma$. Choose a cutoff
$\eta\in C_c^\infty(B_{\Gamma'})$, with $0\le\eta\le1$ and
$\eta=1$ on $B_\Gamma$, and set $\varphi=\eta^4$. Choose $\eta$ so that
\[
|\nabla\varphi|\le C\ell^{-1}\varphi^{3/4},\qquad
|\Delta\varphi|\le C\ell^{-2}\varphi^{1/2}.
\]
The factors in the boundary terms can be written as
\begin{equation}\label{eq:monomials}
\begin{aligned}
u^r|\nabla u|^2&=Q_u^{1/2}u^{t_u/2},&
u^{r-1}|\nabla u|^3&=Q_u^{3/4}u^{t_u/4},\\
u^rv^p|\nabla u|&=Q_u^{1/4}F_1^{1/2}u^{t_u/4},&
u^{r+1}v^p&=F_1^{1/2}u^{t_u/2}.
\end{aligned}
\end{equation}
Here, by Proposition~\ref{prop:circle-coefficients}, \(-2<r,s<0\), hence
\begin{equation}\label{eq:t-range}
t_u=r+2\in(0,2),\qquad t_v=s+2\in(0,2).
\end{equation}

Young's inequality, with the indicated cutoff powers, gives
\begin{equation}\label{eq:young}
\begin{aligned}
u^r|\nabla u|^2|\Delta\varphi|
&\le\epsilon Q_u\varphi+C_\epsilon\ell^{-4}u^{t_u},\\
u^{r-1}|\nabla u|^3|\nabla\varphi|
&\le\epsilon Q_u\varphi+C_\epsilon\ell^{-4}u^{t_u},\\
u^rv^p|\nabla u||\nabla\varphi|
&\le\epsilon(Q_u+F_1)\varphi+C_\epsilon\ell^{-4}u^{t_u},\\
u^{r+1}v^p|\Delta\varphi|
&\le\epsilon F_1\varphi+C_\epsilon\ell^{-4}u^{t_u}.
\end{aligned}
\end{equation}

The same estimates apply to $J_v$. Multiplying
\eqref{eq:constructed-coercivity} by $\varphi$ and integrating
by parts using \eqref{eq:weak-form}, we obtain
\[
\begin{aligned}
&(\varepsilon_0-C_0\epsilon)
\int_{B_{\Gamma'}}
(Q_u+Q_v+F_1+F_2)\varphi\dd\\
&\qquad\le
C_\epsilon(\Gamma'-\Gamma)^{-4}
\int_{B_{\Gamma'}}(u^{t_u}+v^{t_v})\dd,
\end{aligned}
\]
where $C_0$ depends only on $N,p,q$.
Choose $\epsilon>0$ such that
$C_0\epsilon\le\varepsilon_0/2$. Set
\[
\mathcal I(\Gamma)
=\int_{B_\Gamma}(Q_u+Q_v+F_1+F_2)\dd.
\]
Since $\varphi=1$ on $B_\Gamma$ and $\varphi\ge0$, we have
\[
\frac{\varepsilon_0}{2}\mathcal I(\Gamma)
\le
\frac{\varepsilon_0}{2}
\int_{B_{\Gamma'}}
(Q_u+Q_v+F_1+F_2)\varphi\dd
\le
C_\epsilon(\Gamma'-\Gamma)^{-4}
\int_{B_{\Gamma'}}(u^{t_u}+v^{t_v})\dd.
\]
Thus
\begin{equation}\label{eq:annular-ineq}
\mathcal I(\Gamma)
\le C(\Gamma'-\Gamma)^{-4}
\int_{B_{\Gamma'}}(u^{t_u}+v^{t_v})\dd,
\end{equation}
where $C$ depends only on $N,p,q$.

\subsection{Interpolation and iteration}

The lower-power integrals in \eqref{eq:annular-ineq} are controlled by
using the $L^1$ bound \eqref{eq:L1} and the quartic gradient terms.
The Gagliardo-Nirenberg inequality gives a sublinear energy estimate.
Young's inequality and an iteration over nested balls then remove the
dependence on the energy on the outer ball and prove
Proposition~\ref{prop:local-bound}.

If $0<t\le1$, \eqref{eq:L1} and Jensen's inequality give
$\int_{B_{\Gamma'}}u^t\dd\le C$. For $1<t<2$, put $w=u^{t/4}$.
Then
\[
\int_{B_{\Gamma'}}w^{4/t}\dd\le C,\qquad
\int_{B_{\Gamma'}}|\nabla w|^4\dd
=\left(\frac t4\right)^4
\int_{B_{\Gamma'}}u^{t-4}|\nabla u|^4\dd.
\]
On $B_{\Gamma'}$, the Gagliardo-Nirenberg inequality
\cite[Theorem~1]{Nirenberg66} gives
\[
\norm w_{L^4}\le
C\norm{\nabla w}_{L^4}^{\vartheta_t}
\norm w_{L^{4/t}}^{1-\vartheta_t}
+C\norm w_{L^{4/t}},\qquad
\vartheta_t=\frac{N(t-1)}{N(t-1)+4}<1,
\]
where
$1/4=\vartheta_t(1/4-1/N)+(1-\vartheta_t)t/4$.
The constant $C=C(N,t)$ is uniform for $\Gamma'\in[1,2]$.
Consequently,
\[
\int_{B_{\Gamma'}}u^t\dd\le
C\left[1+
\left(\int_{B_{\Gamma'}}u^{t-4}|\nabla u|^4\dd\right)^{\vartheta_t}\right].
\]
The same argument applies to $v$.

 Set $\vartheta_t=0$ when
$t\le1$ and let $\vartheta=\max\{\vartheta_{t_u},\vartheta_{t_v}\}<1$.
Equation~\eqref{eq:annular-ineq} becomes
\[
\mathcal I(\Gamma)\le C(\Gamma'-\Gamma)^{-4}
\bigl(1+\mathcal I(\Gamma')^\vartheta\bigr).
\]
Put $\ell=\Gamma'-\Gamma\in(0,1]$ and
$\gamma=4/(1-\vartheta)$.
For $0<\vartheta<1$, we have
\[
\ell^{-4}=(\ell^{-\gamma})^{1-\vartheta},
\qquad
\ell^{-4}\le\ell^{-\gamma}.
\]
Applying Young's inequality with conjugate exponents
$1/\vartheta$ and $1/(1-\vartheta)$, we obtain, for every
$\epsilon>0$,
\begin{equation}\label{eq:iteration-step}
\begin{aligned}
\mathcal I(\Gamma)
&\le C\ell^{-4}\bigl(1+\mathcal I(\Gamma')^\vartheta\bigr)\\
&=C\ell^{-4}
  +C\mathcal I(\Gamma')^\vartheta
    (\ell^{-\gamma})^{1-\vartheta}\\
&\le C\ell^{-4}
  +\epsilon\mathcal I(\Gamma')
  +\widetilde C_\epsilon\ell^{-\gamma}\\
&\le \epsilon\mathcal I(\Gamma')
  +(C+\widetilde C_\epsilon)\ell^{-\gamma}\\
&=\epsilon\mathcal I(\Gamma')
  +C_\epsilon(\Gamma'-\Gamma)^{-\gamma},
\end{aligned}
\end{equation}
where $C_\epsilon=C+\widetilde C_\epsilon$ depends only on
$N,p,q$ and $\epsilon$.
If $\vartheta=0$, the same conclusion follows directly
with $\gamma=4$ and $C_\epsilon=2C$.

Take $\Gamma_j=2-2^{-j}$ and fix $\epsilon>0$ with
$\epsilon2^\gamma<1$. Iteration yields
\[
\mathcal I(1)\le\epsilon^J\mathcal I(\Gamma_J)
+C_\epsilon2^\gamma\sum_{j=0}^{J-1}(\epsilon2^\gamma)^j.
\]
For each fixed solution, smoothness and positivity on $\overline B_2$
give $\mathcal I(2)<\infty$. The first term on the right tends to zero as
$J\to\infty$, while the geometric series has a bound independent
of the solution. This proves Proposition~\ref{prop:local-bound}.

\subsection{Nonexistence by scaling}\label{sec:technical-scaling}

We show that the common scaling degree of $F_1,F_2$ exceeds $N$
and apply the uniform local bound to the rescaled solution. The estimate obtained by scaling forces the integral over $B_R$
to be zero, which contradicts $u,v>0$.
Theorem~\ref{thm:main} then follows from
Proposition~\ref{prop:reduction}.

The exponents in \eqref{eq:scaling} satisfy
\begin{equation}\label{eq:scale-identities}
\begin{gathered}
p\beta=\alpha+2,\qquad q\alpha=\beta+2,\\
\alpha+\beta-(N-2)
=\frac{(N-2)[G_N+(p-q)^2]}{2(pq-1)}.
\end{gathered}
\end{equation}

\begin{proof}[\normalfont\bfseries Proof of Theorem~\ref{thm:main}]
Proposition~\ref{prop:reduction} reduces the proof to $N\ge5$ and
$(p,q)\in\mathcal R_N$. Suppose a positive entire solution $(u,v)$ exists
for such a pair. Since $p\ne q$ and $G_N\ge0$,
\eqref{eq:scale-identities} gives $\alpha+\beta+2>N$.
The two integrands $F_1=u^rv^{2p},F_2=v^su^{2q}$ have the same homogeneity:
\begin{equation}\label{eq:source-degree}
\begin{aligned}
\nu&=r\alpha+2p\beta=2q\alpha+s\beta\\
&=\tilde{\kappa}(\alpha+\beta+2)+4
>\tilde{\kappa} N+4=N+1.
\end{aligned}
\end{equation}

\Needspace{7\baselineskip}
For every $R>0$, the pair
$u_R(x)=R^\alpha u(Rx)$, $v_R(x)=R^\beta v(Rx)$
solves the same system. Apply \eqref{eq:local-bound} to this
pair and change variables:
\[
\int_{B_R}(F_1+F_2)\dd
\le C(N,p,q)R^{N-\nu}\longrightarrow0
\qquad (R\to\infty).
\]
The left-hand side is positive and nondecreasing in the radius,
whereas the right-hand side tends to zero. This is a contradiction.
\end{proof}

\Needspace{18\baselineskip}
\appendix
\section{Computational data}\label{app:data}

Table~\ref{tab:signs} records the calculation used in
Lemma~\ref{lem:polynomial-bounds}. It lists the coordinate degrees,
the number of Bernstein coefficients and their exact minimum
for each of the four polynomials. The box is
$[0,8/15]\times[0,1]^2$, and the degrees are listed in the order
$(h,y,z)$.

\par\medskip\noindent
\begin{minipage}{\textwidth}
\centering\small
\captionof{table}{Exact positivity certificates. Degrees are ordered as
$(h,y,z)$; the last column is the smallest Bernstein coefficient.}
\label{tab:signs}\label{tab:full-certificate}
\begin{tabular}{@{}lrrr@{}}
\toprule
Polynomial & Degrees & Coefficients & Minimum\\
\midrule
$\widehat\delta_u$ & $(21,5,4)$ & 660 & $4096$\\
$\widehat\delta_v$ & $(17,5,1)$ & 216 & $4096$\\
$R_*$ & $(37,9,4)$ & 1900 & $2^{24}$\\
$H_*$ & $(37,9,4)$ & 1900 & $3\cdot2^{24}$\\
\midrule
Total && 4676 &\\
\bottomrule
\end{tabular}
\end{minipage}
\par\medskip

To reproduce the table, use Python with SymPy~1.14.0. From the
directory containing \texttt{verify.py} and
\texttt{requirements.txt}, run
\begin{quote}\ttfamily
python -m pip install -r requirements.txt\\
python verify.py
\end{quote}
The program computes all 4676 coefficients directly from the
polynomial definitions and \eqref{eq:bernstein-coefficients},
using exact rational arithmetic. For each polynomial, it prints
the three degrees, the coefficient count and the smallest
coefficient, as listed in Table~\ref{tab:signs}.

The complete coefficient tables are also supplied in
\texttt{certificates/source.json} for $R_*,H_*$ and
\texttt{certificates/auxiliary.json} for
$\widehat\delta_u,\widehat\delta_v$. Their keys are
\texttt{R\_star}, \texttt{H\_star}, \texttt{delta\_u\_hat} and
\texttt{delta\_v\_hat}, respectively, matching the names in the
program. These files record the computed values; the program
obtains its coefficients directly from the formulas.

\sloppy
\hbadness=10000

\hbadness=1000
\fussy
\end{document}